\documentclass[a4paper,12pt,reqno]{amsart}
\usepackage{amsmath}
\usepackage{amsfonts}
\usepackage{amssymb}
\usepackage{amsthm}
\usepackage{amscd}
\usepackage{color}
\usepackage{eurosym}
\usepackage{longtable}
\usepackage{verbatim}
\usepackage{eucal,url,amssymb,stmaryrd,enumerate,amscd,}
\usepackage{amsfonts}

\usepackage{amsmath,amsthm,amssymb,amscd,enumerate,eucal,url,stmaryrd}
\usepackage{verbatim}

\usepackage{graphicx}
\usepackage{bm}
\usepackage{enumerate}
\usepackage{pifont}
\usepackage{adjustbox}
\usepackage{tikz}
\usepackage{color}
\usepackage{color}
\usepackage{eurosym}
\usepackage{verbatim}
\usepackage{eucal,url,amssymb,stmaryrd,enumerate,amscd,}
\usepackage{amsfonts}

\newtheorem{theorem}{Theorem}[section]

\newtheorem{proposition}[theorem]{Proposition}
\newtheorem{lemma}[theorem]{Lemma}
\newtheorem*{theorem*}{Theorem}

\theoremstyle{definition}
\newtheorem{definition}[theorem]{Definition}

\newtheorem{remark}[theorem]{Remark}

\newcommand{\beq}{\begin{equation}}
\newcommand{\eeq}{\end{equation}}

\newcommand{\e}{\mathrm{e}}

\newcommand{\SO}{{\mathrm {SO}}}

\newcommand{\GL}{{\mathrm {GL}}}

\newcommand{\G}{{\mathrm G}}

\DeclareMathOperator\ad{ad}

\DeclareMathOperator\vol{vol}

\newcommand{\Ric}{{\rm Ric}}

\newcommand{\frg}{\mathfrak{g}}

\newcommand{\frh}{\mathfrak{h}}
\newcommand{\frk}{\mathfrak{k}}

\newcommand{\frn}{\mathfrak{n}}

\renewcommand{\gg}{\mathfrak{g}}

\newcommand{\gz}{\mathfrak{z}}

\newcommand{\diag}{{\rm diag}}

\numberwithin{equation}{section}

\title[A 7-dimensional nilmanifold with a non Ricci-flat Einstein pseudo-metric]{A non Ricci-flat Einstein pseudo-Riemannian \\
metric on a 7-dimensional nilmanifold}

\author{Marisa Fern\'andez}
\address{Departamento de Matem\'aticas, Facultad de Ciencia y Tecnolog\'{\i}a, Universidad del Pa\'{\i}s Vasco  (UPV / EHU), 
Apartado 644, 48080 Bilbao, Spain}
\email{marisa.fernandez@ehu.es}
\author{Marco Freibert} 
\address{Mathematisches Seminar\\
Christian-Albrechts-Universit\"at zu Kiel\\
Ludewig-Meyn-Strasse 4\\
D-24098 Kiel\\
Germany}
\email{freibert@math.uni-kiel.de}
\author{Jonatan S\'anchez}
\address{Departamento de Matem\'atica Aplicada a las TIC \\ Escuela T\'ecnica Superior de Ingenieros Inform\'aticos UPM\\
Campus de Montegancedo\\
28660 Boadilla del Monte, Madrid, Spain}
\email{jonatan.sanchez@fi.upm.es}

\subjclass[2010]{53C50, 53C25, 53C15, 22E25.}
\keywords{Einstein pseudo-Riemannian metrics, Ricci curvature, nilpotent Lie groups, $\G_2^*$-structures.}
\begin{document}
\begin{abstract} 

We answer in the affirmative the question posed by Conti and Rossi \cite{CR1, CR2} on the existence of
nilpotent Lie algebras of dimension 7 with an Einstein pseudo-metric of nonzero scalar curvature.
Indeed, we construct a left-invariant pseudo-Riemannian metric $g$ of signature $(3, 4)$ on a nilpotent Lie group of dimension 7,
such that $g$ is Einstein and not Ricci-flat.
We show that the pseudo-metric $g$ cannot be induced by any left-invariant closed $\G_2^*$-structure on the Lie group.
Moreover, some results on closed and harmonic $\G_2^*$-structures on an arbitrary 7-manifold $M$
are given. In particular, we prove that the underlying pseudo-Riemannian metric of a closed and harmonic $\G_2^*$-structure on $M$
is not necessarily Einstein, but if it is Einstein then it is Ricci-flat. 
\end{abstract}

\maketitle

%%%%%%%%%%%%%%%%%%%%%%%%%%%%%%%%%%%%%%%%%%%%%%%%%%%%%%%%%%%%%%%%%%%%%%%%%%%%%%%%%%%%%%%%%
%%%%%%%%%%%%%%%%%%%%%%%%%%%%%%%%%%%%%%%%%%%%%%%%%%%%%%%%%%%%%%%%%%%%%%%%%%%%%%%%%%%%%%%							%%%%%%%%%%%%%%%%%%%%%%%%%%%%%%%%%%%%%%%%%%%%%%%%%%%%%%%%%%%%%%%%%%%%%%%%%%%%%%%%%%%%%%%%%
%%%%%%%%%%%%%%%%%%%%%%%%%%%%%%%%%%%%%%%%%%%%%%%%%%%%%%%%%%%%%%%%%%%%%%%%%%%%%%%%%%%%%%%%%
\section{Introduction}\label{sec:intro}
%%%%%%%%%%%%%%%%%%%%%%%%%%%%%%%%%%%%%%%%%%%%%%%%%%%%%%%%%%%%%%%%%%%%%%%%%%%%%%%%%%%%%%%%%
%%%%%%%%%%%%%%%%%%%%%%%%%%%%%%%%%%%%%%%%%%%%%%%%%%%%%%%%%%%%%%%%%%%%%%%%%%%%%%%%%%%%%%%%%

An Einstein manifold $(M, g)$ is a pair consisting of a differentiable manifold $M$ with a
Riemannian, or pseudo-Riemannian, metric $g$ whose Ricci curvature $Ric(g)$ is a multiple of the metric $g$, i.e.
\begin{equation} \label{def:Einstein}
\Ric(g) = \lambda\, g,
\end{equation}
for some real constant $\lambda$. 
The metric $g$ is called {\em Ricci-flat} when $\lambda= 0$,  and 
$g$ is said to be 
{\em Einstein metric of nonzero scalar curvature} (or {\em non Ricci-flat Einstein metric}) when $\lambda\not= 0$.

One of the results of Milnor \cite{M} states that if $G$ is a connected and non-Abelian nilpotent 
Lie group, then for any left-invariant Riemannian metric 
on $G$ there exists a direction of strictly negative Ricci curvature and a direction of strictly positive
Ricci curvature; therefore, such a metric cannot be Einstein.

From now on, we restrict to left-invariant pseudo-Riemannian metrics on nilpotent Lie groups. Such metrics are uniquely determined by
scalar products on the associated Lie algebra, also called pseudo-metrics. The Einstein condition is also expressed by 
\eqref{def:Einstein}.

In contrast to the Riemannian case, there exist Ricci-flat pseudo-metrics on non-Abelian nilpotent Lie algebras 
(see \cite{CalvarusoZaeim:lorentzian, CalvarusoZaeim:neutral, CR1, FLu, Fr, GN, GuB, K2}). Some of these examples are not flat, indicating another difference with the Riemannian case. 
More generally, Conti and Rossi proved in \cite{CR1} that if $\frg$ is a 
nilpotent Lie algebra of dimension $\leq 6$, then
any Einstein pseudo-metric on $\frg$ is Ricci-flat.
This result is also true in dimension 7
under a certain extra condition, e.g. when the Lie algebra is {\em nice} (see \cite{CR1, CR2, LaW} and section \ref{sec:example} for details). In fact, Conti and Rossi 
\cite{CR1} excluded the existence of an Einstein non Ricci-flat pseudo-metric on all 7-dimensional nilpotent Lie algebras with the exception
of 9 Lie algebras and 2 one-parameter families;
and they provided the first example of an Einstein pseudo-metric with nonzero scalar curvature on a nilpotent Lie algebra
of dimension 8 (other examples  of such pseudo-metrics in dimension 8 and in dimension 9 are given in  \cite{CR3}).
 Then they raise the following question 
\begin{center}
{\em  Are there any nilpotent Lie algebras of dimension 7 with an Einstein pseudo-metric of nonzero scalar curvature?}
\end{center}
In this paper, we give an affirmative answer to this question, constructing an Einstein pseudo-metric $g$ of signature $(3, 4)$ 
with nonzero scalar curvature on a seven-dimensional nilpotent Lie algebra $\frg$ (Theorem \ref{th:EinsteinnotRicciflat}). 
If $G$ is the associated simply connected Lie group, the pseudo-metric $g$  on $\frg$ corresponds to a left-invariant 
pseudo-Riemannian metric on $G$, which is Einstein and non Ricci-flat. Since the structure constants are rational 
(see \eqref{def:Lie}), there exists a uniform discrete subgroup $\Gamma$ of $G$;
the quotient $\Gamma{\backslash} G$, called a nilmanifold, has an induced non Ricci-flat Einstein 
pseudo-Riemannian metric (see Remark \ref{rem:nilmfd}).

It seems now interesting to know if the Einstein pseudo-metric $g$ on $\frg$ is induced by 
some well-known geometric structure. A natural class to
look at are $\G_2^*$-structures as these induce pseudo-Riemannian metrics of signature $(3,4)$ (see section \ref{sec:closedG2} for details). 

We prove that the aforementioned Lie algebra $\frg$ has a closed $\G_2^*$-structure but 
the non Ricci-flat Einstein pseudo-metric $g$ is not determined by any closed $\G_2^*$-structure on $\frg$. Indeed, we 
show that no closed $\G_2^*$-structure on $\frg$ can induce a non Ricci-flat  Einstein pseudo-metric (see Theorem~\ref{thm:thmEinstein2}).
The proof is based on some necessary conditions that must satisfy a non Ricci-flat  Einstein pseudo-metric induced by a closed $\G_2^*$-structure on $\frg$
 (see Lemma \ref{lem:grhozero1}, Lemma \ref{lem:grhozero2}, Lemma \ref {lem:grhononzero} and Lemma \ref{lem:grhozero3}).
Theorem~\ref{thm:thmEinstein2} leads to the following question: {\em is there any 7-dimensional nilpotent Lie algebra $\frk$
 with a non Ricci-flat Einstein pseudo-metric associated to a closed $\G_2^*$-structure on $\frk$?}
 
As a consequence of Theorem \ref{thm:thmEinstein2}, and using Nomizu's theorem \cite{Nomizu}, we have that the compact 
nilmanifold $\Gamma{\backslash} G$, considereded above, has a left-invariant closed $\G_2^*$-structure, but no left-invariant 
closed $\G_2^*$-structure on $\Gamma{\backslash} G$ can induce a non Ricci-flat  Einstein left-invariant pseudo-metric. 

We remind that a $\G_2$-{\em structure} on a 7-manifold $M$ 
is characterized by the existence of a globally defined differential 3-form $\psi$ on $M$
satisfying a certain nondegeneracy condition.  Such a 3-form $\psi$ gives rise to a Riemannian metric $g_{\psi}$ and a volume form $vol_{\psi}$ on $M$.
A $\G_2$-structure $\psi$ is said to be {\em closed} if $d\psi=0$, 
while it is called {\em co-closed} if $d\star_{\psi}\psi=0$, where $\star_{\psi}$ is the Hodge operator associated to $g_{\psi}$ and the volume form $vol_{\psi}$ on $M$.  
When both of these conditions hold, the Riemannian metric $g_{\psi}$ is Ricci-flat, 
and $\mathrm{Hol}(g_{\psi})\subseteq\G_2$ (cf.~\cite{Br,FeG}).
 
In the literature, special attention has been given to the case when $g_{\psi}$ is
an Einstein metric \cite{Br, CI, FFeM2, MU}. By the results in \cite{Br, CI} no compact 7-dimensional manifold can support a closed  
$\G_2$-structure $\psi$ whose underlying metric $g_{\psi}$ is Einstein unless $g_{\psi}$ has holonomy contained in $\G_2$. 
The first author in a joint work with Fino and Manero \cite{FFeM} proved that a 7-dimensional solvmanifold cannot admit any left-invariant
closed $\G_2$-structure $\psi$ such that the induced metric $g_{\psi}$ is Einstein, unless $g_{\psi}$ is flat.
Recently, Lotay and Wei proved in \cite[Proposition 1.7]{LW} that no 7-dimensional manifold $M$
can support a closed and harmonic $\G_2$-structure $\psi$ unless 
$\psi$ is also coclosed, and so the underlying metric $g_{\psi}$ has holomomy contained in $\G_2$ and it is Ricci-flat. 

Contrary to the closed $\G_2$-structures case, we prove that 
if $\varphi$ is a closed and harmonic $\G_2^*$-structure on any 7-manifold, then the scalar curvature
of the induced pseudo-Riemannian metric $g_{\varphi}$ vanishes. So, if $g_{\varphi}$ is Einstein, $g_{\varphi}$ is  
Ricci-flat but $\varphi$ is not necessarily coclosed
(Theorem \ref{th:harmonic-Ricci}).
We discuss two examples of seven-dimensional manifolds, each of them with a closed and harmonic (but non-coclosed) $\G_2^*$-structure. 
We show that in one of those examples, the induced pseudo-metric is Ricci-flat, but the pseudo-metric  is not Einstein in the another example.

%%%%%%%%%%%%%%%%%%%%%%%%%%%%%%%%%%%%%%%%%%%%%%%%%%%%%%%
%%%%%%%%%%%%%%%%%%%%%%%%%%%%%%%%%%%%%%%%%%%%%%%%%%%%%%%
%%%%%%%%%%%%%%%%%%%%%%%%%%%%%%%%%%%%%%%%%%%%%%%%%%%%%%%
%%\section{Non Ricci-flat and Einstein pseudoriemannian 7-nilmanifold }\label{sec:example}
\section{The 7-dimensional nilpotent Lie algebra $\frg$}\label{sec:example}
%%%%%%%%%%%%%%%%%%%%%%%%%%%%%%%%%%%%%%%%%%%%%%%%%%%%%%%
%%%%%%%%%%%%%%%%%%%%%%%%%%%%%%%%%%%%%%%%%%%%%%%%%%%%%%%

In this section we construct an explicit non Ricci-flat Einstein pseudo-metric on a 
nilpotent Lie algebra $\frg$ of dimension 7. 

Before constructing our metric, it is worth recalling the notion of a {\em nice} nilpotent Lie algebra 
introduced by Lauret and Will in \cite{LaW}. Let $\{e_1,\dots,e_n\}$ be a basis for a nilpotent Lie algebra 
${\mathfrak{h}}$, with structural constants $a^k_{ij}$. 
We say that the basis $\{e_i\}$ is {\em nice} if the following conditions hold:
\begin{itemize}
	\item for all $i < j$ there is at most one $k$ such that $a^k_{ij}\neq 0$;
	\item if $a^k_{ij}$ and $a^k_{lm}$ are nonzero then either $\{i, j\} = \{l,m\}$ or $\{i, j\}\cap\{l,m\}=\emptyset$.
\end{itemize}
A nilpotent Lie algebra is said to be {\em nice} if it admits a nice basis. 
\begin{proposition}{\rm{(}}\cite[Theorem 4.3, Theorem 5.5]{CR1}{\rm{)}}
If $\frh$ is either a nilpotent Lie algebra of dimension $\leq 6$, or a nice nilpotent Lie algebra of dimension 7, then 
any Einstein pseudo-metric on $\frh$ is Ricci-flat.
\end{proposition}

Note that a classification of 7-dimensional indecomposable nilpotent Lie algebras was given
by Gong in \cite{Gong} (see also \cite{CFe}). This list contains 140 Lie algebras and 9 one-parameter families.
In addition,  there are 35 decomposable nilpotent Lie algebras (\cite{Magnin, Salamon1}).

Actually, Conti and Rossi were able to exclude the existence of an Einstein pseudo-metric  with nonzero scalar curvature for a 
wide class of 7-dimensional nilpotent Lie algebras, namely for all
but $9$ indecomposable Lie algebras and 2 one-parameter families (see \cite[Theorem 4.4]{CR1}).
These Lie algebras $\frh$ are characterized by the condition that all derivations have zero trace,
i.e. by $\mathrm{Der}(\frh)\subseteq\mathfrak{sl}(\frh)$ (see \cite[Theorem 4.1]{CR1}). One of those $9$ nilpotent  Lie algebras
is the Lie algebra $\frg$ defined as follows:
\begin{equation}\label{def:Lie}
\frg = (0,0,e^{12},e^{13},e^{14},e^{15}+e^{23},e^{16}+e^{23}+e^{24}).
\end{equation}
By this notation we mean that the dual space $\gg^\ast$ has a fixed coframe
$\{e^{1}, \dots, e^{7}\}$ such that
\begin{equation*}
\begin{aligned}
& de^{1} = de^{2} =0, \quad \,\, de^{3} =e^{12}, \quad de^{4} =e^{13}, \quad de^{5} =e^{14}, \\
&  de^{6} =e^{15}+e^{23},  \quad de^{7} =e^{16}+e^{23}+e^{24},
\end{aligned}
\end{equation*}
where $d$ denotes the Chevalley-Eilenberg differential on $\gg^\ast$, and $e^{12}$ stands for $e^1 \wedge e^{2}$, and so on.

\begin{theorem}\label{th:EinsteinnotRicciflat}
The nilpotent Lie algebra $\frg$ defined in \eqref{def:Lie} has an Einstein pseudo-metric $g$
with nonzero scalar curvature.
\end{theorem}
\begin{proof}
Let us consider the coframe $\{f^1, \ldots, f^7\}$ of $\gg^\ast$ that, for simplicity of the expressions, we define by 
\begin{equation}\label{eq:changeofbasis}
\begin{aligned}
& e^1= f^1,  \quad \,\, e^2=f^2, \quad \,\, e^3= - \frac{26}{51}\,f^2 + f^3, \\
& e^4=-\frac{2300}{2601}\,f^2 - \frac{26}{51}\,f^3 + f^4,  \quad \,\,  e^5= -\frac{2300}{2601}\,f^3 + \frac{46}{51}\,f^4 + f^5, \\
& e^6= -\frac{2300}{2601}\,f^4 + f^6, \quad \,\, e^7= -\frac{50}{51}\,f^6 + f^7.
\end{aligned}
\end{equation}

In this new coframe, the structure equations of $\frg$ are
\[
\left(0,0,f^{12}, \, f^{13},\, f^{14}-\frac{24}{17}f^{13},\,\, \frac{46}{51} f^{14}+f^{15}+f^{23},\,\, \frac{50}{51} f^{15}+f^{16}+\frac{25}{17}f^{23}+f^{24}\right).
\]
We define the pseudo-metric $g$ on $\frg$, that with respect to the dual frame $\{f_1, \ldots, f_7\}$ of $\frg$ is given by the 
diagonal matrix
\begin{multline} \label{def:einsteinmet}
g = \diag\Bigl(
 \frac{71639296000000000}{168377826559400929}, \, -\frac{1946720000000}{2015993900449},\\ -\frac{2116000000}{6975757441},\, \frac{21160000}{24137569},\, -\frac{115000}{250563},\,-\frac{600 }{289}, \,1 \Bigr).
\end{multline}
One can check that $g$ is Einstein but not Ricci-flat. Indeed, with respect to the orthogonal frame $\{f_1, \ldots, f_7\}$, we have 
\[
\mathrm{Ric}(g)=\frac{48661191875666868481}{659081523200000000000}\, g.
\]
Moreover, for the scalar curvature $\mathrm{scal}(g)$ of $g$ we obtain $\mathrm{scal}(g)= -\frac{48661191875666868481}{659081523200000000000}$ since $\mathrm{trace}(g) = -1$.
\end{proof}

\begin{remark}\label{rem:nilmfd}
 Let $G$ be the connected, simply connected and nilpotent Lie group of dimension 7 with Lie algebra $\frg$. By \eqref{def:Lie},
the structure constants of $\frg$ are rational numbers. Thus, Malcev's theorem \cite{Mal} implies that there exists a uniform discrete subgroup
$\Gamma$ of $G$ such the quotient space $\Gamma{\backslash} G$,
called a nilmanifold, is compact. The pseudo-metric $g$ on $\frg$ defines a left-invariant 
pseudo-Riemannian metric, that we denote also by $g$, on the Lie group $G$ such that $g$ is Einstein and non Ricci-flat.
Hence, $g$ descends to a non Ricci-flat Einstein pseudo-Riemannian metric on $\Gamma{\backslash} G$, because 
the projection $\pi \colon G \longrightarrow \Gamma{\backslash} G$ is a local isometry, and so $\pi$ preserves the curvature.
\end{remark}

\begin{remark}
We would like to notice that to find the pseudo-metric $g$ defined in \eqref{def:einsteinmet}, we use  {\em Mathematica}$^{\text{©}}$  to 
determine a frame $\{f_1, \ldots, f_7\}$ of the Lie algebra $\frg$ so that 
$\{f_1, \ldots, f_7\}$ is orthogonal with respect to $g$ and that $g$ meets the required conditions. 
To this end, if $\{e_1, \ldots, e_7\}$ is the frame of $\frg$ dual to the coframe $\{e^1, \ldots, e^7\}$,
 we assume that each vector $e_i$ $(1 \leq i \leq 7)$ is non-isotropic under the 
desired pseudo-metric $g$. Thus, a Gram-Schmidt process can be performed. After this process, we obtain a $g$-orthogonal
frame $\{f_1,\ldots,f_7\}$ of $\frg$ such that
\[
f_k=\sum_{i=1}^k p_k^i e_i\,, 
\]
with  $p_k^k=1$, for all  $k=1,\ldots,7$. 
Therefore, we have reduced from  77 to 28 the number of unknowns to be determined. 
Then, we determine the Ricci curvature $\mathrm{Ric}(g)$ of $g$ with respect to the frame $\{f_1,\ldots,f_7\}$. 
By imposing that $\mathrm{Ric}(g)$ must be diagonal, i.e. $\mathrm{Ric}(g)(f_i,f_j)=0$ whenever $i\neq j$, we obtain a system of~$21$ equations. 
Our strategy to solve that system is as follows.  We take a small subset of  equations, which involve a reduced number of 
 unknowns, and we determine the value of some of these unknowns (usually this value will be  given in terms of the other unknowns). 
We repeat this process until the system is solved.  Then, we discard those solutions that  imply
 that  $g(f_i,f_i)=0$, for some $i\in \{1, \ldots, 7\}$, as well as those solutions in which some element of the diagonal of the Ricci curvature vanishes.  
Finally, for each valid solution, we must check if it determines a non Ricci-flat Einstein pseudo-metric of signature $(3, 4)$.
\end{remark}

%%%%%%%%%%%%%%%%%%%%%%%%%%%%%%%%%%%%%%%%%%%%%%%%%%%%%%%
%%%%%%%%%%%%%%%%%%%%%%%%%%%%%%%%%%%%%%%%%%%%%%%%%%%%%%%
%%%%%%%%%%%%%%%%%%%%%%%%%%%%%%%%%%%%%%%%%%%%%%%%%%%%%%%
\section{Closed $\G_2^*$-structures on $\frg$}\label{sec:closedG2}
%%%%%%%%%%%%%%%%%%%%%%%%%%%%%%%%%%%%%%%%%%%%%%%%%%%%%%%
%%%%%%%%%%%%%%%%%%%%%%%%%%%%%%%%%%%%%%%%%%%%%%%%%%%%%%%

In this section, we show that the Lie algebra $\frg$ defined in \eqref{def:Lie} has a closed 
$\G_2^*$-structure, but $\frg$ cannot support any closed $\G_2^*$-structure  whose underlying pseudo-metric
is Einstein and non Ricci-flat. Therefore, the non Ricci-flat Einstein pseudo-metric given in \eqref{def:einsteinmet} is not determined by any 
closed $\G_2^*$-structure on $\frg$.
 First, we need some definitions and results about $\G_2^*$-structures. 
(For more details on the group $\G_2^*$ and $\G_2^*$-structures see \cite{K1, FK, FLu}.)

Let $V$ be a real vector space of dimension 7. Consider the representation of the general linear group ${\GL}(V)$ on 
the space $\Lambda^k(V^*)$ of $k$-forms on $V$. An element $\rho\in\Lambda^k(V^*)$ is said to be \emph{stable} if its orbit 
under ${\GL}(V)$ is open in $\Lambda^k(V^*)$. In order to describe the open orbits in $\Lambda^3(V^*)$ we proceed as follows. 
For each 3-form $\varphi\in\Lambda^3(V^*)$ on $V$, we can define the symmetric quadratic form
 $b_\varphi: V\times V \to \Lambda^7(V^*)$ by
\begin{equation} \label{def:b}
6\,b_\varphi(v,w)=\iota_v\varphi\wedge\iota_w\varphi\wedge\varphi,
\end{equation}
where $v, w\in V$, and $\iota_v$ denotes the contraction by the vector $v \in V$. By \cite{Hit}, 
$\varphi$ is \emph{stable}, i.e. the ${\GL}(V)$-orbit of $\varphi$ is open in $\Lambda^3 (V^*)$, if and only if $b_{\varphi}$ is non-degenerate. In this case, we consider 
the symmetric map $g_{\varphi}: V\times V \to {\mathbb{R}}$ given by
\begin{equation}\label{eq:metric7-dim}
g_{\varphi}(v,w)\, \mathrm{vol}_{\varphi}=b_{\varphi}(v,w),
\end{equation}
where $v, w\in V$, and $\mathrm{vol}_{\varphi}$ is the volume form given by 
$$
\mathrm{vol}_{\varphi}=\sqrt[9]{\mathrm{det}(b_\varphi)}.
$$
Hitchin shows in \cite{Hit} that there are exactly two ${\GL}(V)$ open orbits in $\Lambda^3(V^*)$ which are defined by
$$
\Pi_+(V^*)=\left\{\varphi\in\Lambda^3(V^*) \;|\;g_\varphi\; \text{is positive definite}\right\},
$$
and
$$
\Pi_-(V^*)=\left\{\varphi\in\Lambda^3(V^*) \;|\;g_\varphi  \;\text{is indefinite}\right\}.
$$
Moreover, if $\varphi\in\Pi_{+}(V^*)$, then $\varphi$ defines a $\G_2$-{\em structure} on $V$. But if
$\varphi\in\Pi_-(V^*)$, then $\varphi$ defines a $\G_{2}^*$-{\em structure}  on $V$. In this case,
the metric $g_\varphi$ has signature $(3,4)$ and the stabiliser of $\varphi$ is the non-compact group $\G_2^* \subset \SO(3,4)$, i.e.
$$
\G_2^* = \{A \in {\GL}(V) \mid A^* \varphi  = \varphi\}.
$$ 
If $\varphi\in\Pi_-(V^*)$, there exists a $g_\varphi$-orthonormal
coframe $\{v^1, \dots, v^7\}$ of $V^*$ with $g_{\varphi}(v^i,v^i)=-1$ for $i=1,2,3,4$, and $g_{\varphi}(v^j,v^j)=1$ for $j=5,6,7$, such that the form $\varphi$ is given by 
\begin{equation}\label{def G2*}
\begin{array}{l}
\varphi= - v^{127}-v^{347}+v^{567}+v^{135}-v^{146}-v^{236}-v^{245},\\
\end{array}
\end{equation}
where $v^{127}$ stands for $v^1 \wedge v^2 \wedge v^3$, and so on. The action of $\G_2^*$ on $V$ does not only preserve the form $\varphi$ but also the metric
$$
g_{\varphi} = - (v^1)^2 - (v^2)^2 - (v^3)^2 - (v^4)^2 + (v^5)^2 + (v^6)^2 + (v^7)^2,
$$
and the volume form
$$
\mathrm{vol}_{\varphi} = v^1 \wedge v^2 \wedge v^3  \wedge v^4  \wedge v^5  \wedge v^6  \wedge v^7.
$$
Denote by $\star_{\varphi}$ the Hodge star operator determined by $g_{\varphi}$ and the volume form $\mathrm{vol}_{\varphi}$. Then the Hodge dual $\star_{\varphi}\varphi$
is the 4-form given by
\begin{equation}\label{eqn:star-varphi}
\star_{\varphi}\varphi= v^{1234}-v^{1256} - v^{3456} - v^{2467}+v^{2357}+v^{1457}+v^{1367}.
\end{equation}
Note that $\G_2^*$ also preserves the form $\star_{\varphi}\varphi$. 

Conversely, if $\varphi$ is a 3-form on $V$ such that there exists a coframe 
$\left\{v^1, \dots, v^7\right\}$ of $V^*$ for which $\varphi$ is expressed as in \eqref{def G2*},
then $\varphi$ is a $\G_2^*$-structure on $V$.
Such a coframe $\left\{v^1, \dots, v^7\right\}$ is then called \emph{adapted to the $\G_2^*$-structure} $\varphi$.

A  $\G_2^*$-{\em structure} on a 7-dimensional smooth manifold $M$ 
is a reduction of the structure group of 
its frame bundle from ${\GL}(7,\mathbb{R})$ to the  group $\G_2^*$.  

The presence of a $\G_2^*$-structure is equivalent to the existence of a special kind of differential 3-form
$\varphi$ on $M$, which can be defined as follows. Denote by $T_{p}(M)$ the tangent space to $M$ at $p\in M$, and by $\Omega^*(M)$ the
algebra of the differential forms on $M$.

\begin{definition} 
Let $M$ be a smooth manifold of dimension 7. A $\G_2^*$-structure on $M$ consists 
of a differential 3-form $\varphi\in \Omega^3(M)$ such that,  for each point $p\in M$, 
$\varphi_{p}$ is a $\G_2^*$-structure on the vector space $T_{p}(M)$. 
\end{definition}

Therefore, if $\varphi$ defines a $\G_2^*$-structure on $M$, then $\varphi$
can locally be written as \eqref{def G2*} with respect to some local basis $\{v^1, \ldots , v^7\}$  of local 1-forms on $M$. 

Since $\G_2^* \subset \SO(3,4)$, a $\G_2^*$-structure $\varphi$ determines a pseudo-Riemannian metric $g_{\varphi}$ of signature $(3, 4)$
and an orientation on $M$, pointwise defined as explained above. 
Denote by $\star_{ \varphi}$  the Hodge star operator determined by $g_{\varphi}$ and the orientation induced by $g_{\varphi}$.

\begin{definition}
We say that a manifold $M$ has a {\em closed} $\G_2^*$-{\em structure}
if there is a $\G_2^*$-structure $\varphi$ on $M$ such that $\varphi$ is closed, that is 
$d\varphi=0$. A manifold $M$ has a {\em coclosed} $\G_2^*$-{structure}
if there is a $\G_2^*$-{structure} $\varphi$ on $M$ such that $\varphi$ is coclosed, i.e. $d(\star_{\varphi}\varphi)=0$.
\end{definition}

\begin{proposition}\rm(\cite{Br2, FeG}\rm)
Let $M$ be a 7-manifold with a $\G_2^*$-structure $\varphi$, and let $\nabla$ be the
Levi-Civita connection of the associated  pseudo-Riemannian metric $g_{\varphi}$.  The
following conditions are equivalent:
\begin{enumerate}
\item[(i)]  $\nabla \varphi=0$;

\item[(ii)] $d\varphi=0$ and $d(\star_{\varphi}\varphi)=0$;
\end{enumerate}
Both conditions imply that the holonomy group of $g_{\varphi}$ is isomorphic to a subgroup
of $\G_2^*$.
\end{proposition}

Now let $G$ be a 7-dimensional simply connected nilpotent
Lie group with Lie algebra  
$\mathfrak g$. Then a $\G_2^*$-structure on $G$ 
 is \emph{left-invariant} if and only if the corresponding
3-form $\varphi$ is left-invariant. Thus, a left-invariant $\G_2^*$-structure on $G$ 
 corresponds to an element $\varphi$ of $\Lambda^3({\mathfrak g}^*)$ 
that can be written as \eqref{def G2*} 
with respect to some $g_\varphi$-orthonormal coframe $\{v^1,\dotsc, v^7\}$ of the dual space 
${\mathfrak{g}}^*$ of $\mathfrak g$. Such an element 
$\varphi \in \Lambda^3({\mathfrak g}^*)$ defines a $\G_2^*$-{\em structure} on 
$\mathfrak g$. Moreover, with respect to the coframe $\{v^1,\dotsc, v^7\}$ of 
${\mathfrak{g}}^*$,  the dual form $\star_{\varphi}\varphi$ is expressed as \eqref{eqn:star-varphi}.
If $d \varphi =0$,  then the $\G_2^*$-structure on 
$\mathfrak g$ is called {\it closed}, and 
if $\varphi$ is coclosed, that is if $\star_{ \varphi} \varphi$ is closed, then the $\G_2^*$-structure is called {\it coclosed}.
 
With this background in mind, we will prove in Theorem~\ref{thm:thmEinstein2} that no closed $\G_2^\ast$-structure on 
the Lie algebra $\frg$, defined in \eqref{def:Lie}, 
can induce a non Ricci-flat Einstein pseudo-metric, and so no closed $\G_2^\ast$-structure on $\frg$
induces the pseudo-metric $g$ given in
 \eqref{def:einsteinmet}. To this end, we need to first show some  
 necessary conditions that must satisfy a pseudo-metric induced by a closed $\G_2^*$-structure on $\mathfrak g$.

From now on, we denote by $\{e_{1}, \dots, e_{7}\}$ the frame of  $\frg$ dual to the 
coframe $\{e^{1}, \dots, e^{7}\}$ of $\gg^\ast$ (see  \eqref{def:Lie}). Moreover, if
$\rho$ is a $\G_2^\ast$-structure on the Lie algebra $\frg$,
we identify $b_{\rho}\in S^2 \frg^*\otimes \Lambda^7 \gg^\ast$ (where $b_{\rho}$ is defined in  \eqref{def:b}) with an element in $S^2\gg^\ast$
 using the volume form $e^{1234567}$ of $\frg$.

\begin{lemma}\label{lem:grhozero1}
Let $\frg$ be the 7-dimensional Lie algebra defined in \eqref{def:Lie}, and let $\rho$ be a closed $\G_2^\ast$-structure on $\frg$  inducing the pseudo-metric $g_\rho$. 
Then, $g_\rho(e_7,e_7)$, $g_\rho(e_7,e_6)$, $g_\rho(e_7,e_5)$ and $g_\rho(e_6,e_6)$ are all zero.
\end{lemma}
\begin{proof}
From~\eqref{eq:metric7-dim}, we know that if $u, v\in \frg$ then, $g_{\rho}(u,v)=0$ if and only if $b_\rho(u,v)=0$, 
where $b_\rho(u,v)= \frac{1}{6} \iota_u\rho\wedge\iota_v\rho\wedge\rho$ by \eqref{def:b}. 
To prove that $b_{\rho}(e_7,e_7)=0$, $b_{\rho}(e_7,e_6)=0$, $b_{\rho}(e_7,e_5)=0$ and  $b_{\rho}(e_6,e_6)=0$ we proceed as follows.
A generic closed 3-form $\gamma$ on $\gg$ has the following expression
\begin{equation}
\begin{aligned}\label{eq:closedform}
& \gamma=\, c_{123}\, e^{123}+c_{124} \,e^{124}+c_{125} \,e^{125}+c_{126} \,e^{126}+c_{127}\, e^{127}+c_{134}\, e^{134}+c_{135} \,e^{135}\\
& \,\,\,\,\, + c_{136} \,e^{136}+c_{137} \,e^{137}+c_{145} \,e^{145}+c_{146\,}e^{146}+c_{147} \,e^{147}+c_{156} \,e^{156}+c_{157}\, e^{157}\\
&\,\,\,\,\,+ c_{167}\, e^{167}+c_{234} \,e^{234}+(c_{137} - c_{146} - c_{147}) \,e^{235}+c_{236} \,e^{236} +c_{237} \,e^{237} \\
& \,\,\,\,\,+ (-c_{156} - c_{157}-c_{236}) \,e^{245}+(- c_{167}-c_{237})\,e^{246} + c_{167} \,e^{247} - 2\, c_{167} \,e^{256} \\
& \,\,\,\,\, + (-c_{157} + c_{167}+c_{237}) \,e^{345},
\end{aligned}
\end{equation}
where $c_{ijk}$ are arbitrary real numbers. Thus, 
\begin{equation*}
\iota_{e_7}\gamma=c_{127} \,e^{12}+c_{137} \,e^{13}+c_{147}\,e^{14}+c_{157} \,e^{15}+c_{167} \,e^{16}+c_{237} \,e^{23}+ c_{167} \,e^{24}.
\end{equation*}
Consequently,
\begin{equation*}
\iota_{e_7}\gamma \wedge \iota_{e_7}\gamma\wedge \gamma=0,
\end{equation*}
for any closed 3-form $\gamma$ on $\frg$. 
This equality and  \eqref{def:b} imply that if $\rho =  \gamma$ is a closed $\G_2^*$-structure on $\frg$, then $b_{\rho}(e_7,e_7)=0$.

Using again \eqref{eq:closedform} we have 
\begin{equation}\label{eq:e6-0}
\begin{aligned}
& \iota_{e_6}\gamma=c_{126} \,e^{12}+c_{136} \,e^{13}+c_{146}\,e^{14}+c_{156} \,e^{15}-c_{167} \,e^{17}+c_{236} \,e^{23} \\
& \,\,\,\,\, \,\,\,\,\, \,\,\,\, - (c_{167} +c_{237}) e^{24} - 2 c_{167} \,e^{25}.
\end{aligned}
\end{equation}
So,
\begin{equation}\label{eq:e6-1}
\iota_{e_7}\gamma \wedge \iota_{e_6}\gamma\wedge \gamma=0, 
\end{equation}
and 
\begin{equation}\label{eq:e6-2}
\iota_{e_6}\gamma \wedge \iota_{e_6}\gamma\wedge \gamma=0,
\end{equation}
for any closed 3-form $\gamma$ on $\frg$. 

Now, let $\rho =  \gamma$ be a closed $\G_2^*$-structure on $\frg$. Then, from  \eqref{def:b} and  \eqref{eq:e6-1} we obtain $b_{\rho}(e_7,e_6)=0$.
Moreover, \eqref{def:b} and  \eqref{eq:e6-2} imply that $b_{\rho}(e_6,e_6)=0$.

Using  \eqref{eq:closedform} we have that the 2-form $\iota_{e_5}\gamma$ on $\frg$ is given by
\begin{equation}\label{eq:e5}
\begin{aligned}
& \iota_{e_5}\gamma=c_{125} \,e^{12}+c_{135} \,e^{13}+c_{145}\,e^{14} - c_{156} \,e^{16}-c_{157} \,e^{17} - (c_{156} + c_{157} + c_{236}) \,e^{24} \\
& \,\,\,\,\,\,\,\,\, \,\,\, + 2\, c_{167} \,e^{26}  +   (- c_{157} + c_{167} + c_{237}) \,e^{34}.
\end{aligned}
\end{equation}
Then,
\begin{equation*}\label{eq:e6-3}
\iota_{e_7}\gamma \wedge \iota_{e_5}\gamma\wedge \gamma=0,
\end{equation*}
for any closed 3-form $\gamma$ on $\frg$. Therefore, if $\rho =  \gamma$ is a closed $\G_2^*$-structure on $\frg$, we have 
$b_{\rho}(e_7,e_5)=0$.
\end{proof}

\begin{lemma}\label{lem:grhozero2}
Let $\frg$ be the 7-dimensional Lie algebra defined in \eqref{def:Lie}, and let $\rho$ be a closed $\G_2^\ast$-structure on $\frg$  inducing the pseudo-metric $g_\rho$. 
Then, $g_\rho(e^1,e^1)$ and $g_\rho(e^1,e^2)$ are both zero.
\end{lemma}

\begin{proof}
Clearly $\rho$ can be expressed as in~\eqref{eq:closedform} since $\rho$ is a closed 3-form on $\frg$.
Then, using ~\eqref{def:b} and~\eqref{eq:metric7-dim} we can determine $g_\rho(e_i,e_j)$, for $1 \leq i \leq j \leq 7$.
Let $A=(g_\rho(e_i,e_j))_{i,j=1}^7$ be the Gram matrix of the pseudo-metric $g_\rho$.
Then, $A^{-1}=(g_\rho(e^i,e^j))_{i,j=1}^7$ is the Gram matrix of $g_\rho$ on $\gg^\ast$ with respect to the basis $\{e^1,\ldots,e^7\}$ of  $\gg^\ast$.
 Therefore, $g_\rho(e^1,e^1)$ is a multiple of 
\[
\frac{\det\left(b_\rho(e_i,e_j)\right)_{i,j=2}^7}{\det \left(b_\rho(e_i,e_j)\right)_{i,j=1}^7}.
\]
But a calculation (by using a symbolic software) shows that $\det\left(b_\rho(e_i,e_j)\right)_{i,j=2}^7=0$.
Hence, $g_\rho(e^1,e^1)=0$.
Similarly, one can check that $\det\left(b_\rho(e_i,e_j)\right)_{i \not=1, j\not=2}=0$, and so $g_\rho(e^1,e^2)=0$.
\end{proof}

In order to show other conditions that must satisfy a non Ricci-flat Einstein pseudo-metric $g_\rho$, determined 
by a closed $\G_2^\ast$-structure on $\frg$, we need the following result, which shows an explicit expression 
of the Ricci curvature of any pseudo-metric on a unimodular Lie algebra with Killing form zero. 

\begin{proposition} {\rm{(}}\cite[Proposition 2.1]{CR1}{\rm{)}} \label{pro:obstruction1}
Let $\frh$ a unimodular Lie algebra with Killing form zero and a pseudo-metric $h$.
Let $\ad(\frh)$ be the image of $\frh$ in $\gz(\frh)^\circ\otimes \frh^1,$
where $\frh^1 = [\frh, \frh]$, and $\gz(\frh)^\circ$ is the annihilator of the center $\gz(\frh)$ of $\frh$ in $\frh^*$. Then,
\begin{equation}\label{eq:Riccicurvature}
\Ric(h)(u,v)=\frac12\Big(h(du^\flat,dv^\flat)-h(\ad(u),\ad(v))\Big),
\end{equation}
for $u, v \in \frh$, and  where $\flat: \frh\to {\frh}^\ast$ denotes the musical isomorphism induced by the pseudo-metric $h$.
\end{proposition}

Note that the pseudo-metric $h$ on $\frh$ induces not only a pseudo-metric $h$ on $\frh^\ast$ but also 
a pseudo-metric $h$ on $\Lambda^{k}(\frh^\ast)$ which is given by
$$
h(\alpha^1\wedge \ldots \wedge \alpha^k,\, \beta^1\wedge \ldots \wedge \beta^k) = \det\Big(h(\alpha^i,\beta^j)\Big)_{i, j = 1} ^{k},
$$
for $\alpha^1, \ldots, \alpha^k, \beta^1, \ldots,\beta^k \in \frh^\ast$. Moreover, since $\ad(\frh) \subset \gz(\frh)^\circ\otimes \frh^1 \subset \frh^\ast \otimes \frh$, 
the pseudo-metric $h$ on $\frh$ induces a pseudo-metric $h$ on $\ad(\frh)$ which is given by
$$
h(\alpha\otimes u, \beta\otimes v) = h(\alpha, \beta) \, h(u,v), 
$$
for $\alpha\otimes u, \, \beta\otimes v \in \ad(\frh)$.

Now, let $\frg$ be the 7-dimensional Lie algebra defined in \eqref{def:Lie}, and let $\rho$ be a closed $\G_2^\ast$-structure 
on $\frg$  inducing the pseudo-metric $g_\rho$. As a consequence of Proposition \ref{pro:obstruction1}, we have
\begin{equation}\label{eq:Riccicurvature-lie}
\Ric(g_\rho)(e_7, v)=\frac12 \, g_\rho \left(de_{7}^\flat, dv^\flat\right),
\end{equation}
for any $v\in \frg$. In fact,  by ~\eqref{def:Lie}, we know that $\gz(\gg) = \langle e_7 \rangle$,
where $\gz(\gg)$ denotes the center of $\gg$.
So, $\ad(e_7)=0$. Therefore, from~\eqref{eq:Riccicurvature},
$$
Ric(h)(e_7, v)=\frac12 \, h \left(de_{7}^\flat, dv^\flat\right),
$$ 
for any $v\in \frg$, and for any pseudo-metric $h$ on $\gg$. In particular, for $h=g_{\rho}$, we have~\eqref{eq:Riccicurvature-lie}.

\begin{lemma} \label{lem:grhononzero}
Let $\frg$ be the Lie algebra defined in \eqref{def:Lie}, and let $\rho$ be a closed $\G_2^\ast$-structure on $\frg$ whose underlying pseudo-metric
$g_\rho$  is Einstein and non Ricci-flat. Then, $g_\rho(e_4,e_7)$ and $g_\rho(e_5,e_6)$ are both nonzero.
\end{lemma}

\begin{proof}
Assume that $g_\rho(e_4,e_7)=0$ or, equivalently, $b_\rho(e_4,e_7)=0$. 
Then we see that $g_\rho(e_3,e_7)$ also vanishes. In fact, since $\rho$ is a closed 3-form on $\frg$, the 3-form $\rho$ can be expressed as in~\eqref{eq:closedform}.
Proceeding as in the proof of Lemma~\ref{lem:grhozero1}, we obtain 
\[
b_\rho(e_4,e_7)=\frac12\, c_{167}^2 \left(-c_{157}+c_{167}+c_{237}\right),
\]
and 
\[
b_\rho(e_3,e_7)=\frac12 \,c_{167}\, c_{237} \left(-c_{157}+c_{167}+c_{237}\right).
\]
Thus, $b_\rho(e_3,e_7)=0$ or, equivalently, $g_\rho(e_3,e_7)=0$ if $g_\rho(e_4,e_7)=0$.

Moreover, from Lemma~\ref{lem:grhozero1} we know that $g_\rho(e_i, e_7)=0$, for 
$i=5,6,7$. Consequently, if $g_\rho(e_4,e_7)=0$ (and so $g_\rho(e_3,e_7)=0$)  we have $e_7^\flat = \iota_{e_7} g_\rho = g_\rho(e_1,e_7)\, e^1 + g_\rho(e_2,e_7) \,e^2.$ 
Hence, 
$$
de_7^\flat= g_\rho(e_1,e_7)\, de^1 + g_\rho(e_2,e_7) \,de^2 =0,
$$
where the last equality is due to  \eqref{def:Lie}. 
But, by~\eqref{eq:Riccicurvature-lie}, $d(e_7^\flat)=0$ implies
$$
\Ric(g_\rho)(e_7, v)=0,
$$ 
for any $v\in \gg$. Now, using that $g_\rho$ is a non Ricci-flat Einstein pseudo-metric,
i.e. $\Ric(g_\rho)=\lambda\, g_\rho$ with $\lambda$ a nonzero constant, we conclude that 
$g_\rho(e_7,v)=0$, for any $v\in \gg$. But this contradicts that $g_\rho$ is the pseudo-metric induced by a $\G_2^\ast$-form.
Therefore, 
$$
g_\rho(e_4,e_7)\not=0. 
$$
 Moreover, since $\rho$ is a closed 3-form on $\gg$, from  \eqref{eq:closedform},  \eqref{eq:e6-0} and \eqref{eq:e5}, we have
$$
b_\rho(e_5,e_6)= - c_{167}^2 \left(-c_{157}+c_{167}+c_{237}\right)=-2\,b_\rho(e_4,e_7),
$$
which implies that $b_\rho(e_5,e_6)$ or, equivalently, $g_\rho(e_5,e_6)$ is nonzero because $g_\rho(e_4,e_7)$ is nonzero.
\end{proof}

\begin{lemma}\label{lem:grhozero3}
Let $\rho$ be a closed $\G_2^\ast$-form on $\gg$ such that the induced pseudo-metric $g_\rho$ is  Einstein and non Ricci-flat. Then, $g_\rho(e^1,e^3)$ and $g_\rho(e^1,e^4)$ must be zero.
\end{lemma}
\begin{proof}
Firstly we prove that $g_\rho(e^1,e^3)=0$.
From Lemma~\ref{lem:grhozero1} we know that $g_\rho(e_7,e_7)=0$. Thus, $\Ric(g_\rho)(e_7,e_7)=0$
since the pseudo-metric $g_\rho$ is Einstein.
Then, by \eqref{eq:Riccicurvature-lie},
\begin{equation}\label{eqn:ric7}
g_{\rho}(d e_7^\flat, de_7^\flat) =0.
\end{equation}
On the other hand, we determine $d e_7^\flat$ as follows. We know that $g_\rho(e_i, e_7)=0$ $(i=5,6,7)$ by Lemma~\ref{lem:grhozero1}, but $g_\rho(e_4,e_7)\not=0$
by Lemma~\ref{lem:grhononzero}. Hence, $e_7^\flat\ = \iota_{e_7} g_\rho = g_\rho(e_1,e_7)\, e^1 + g_\rho(e_2,e_7) \,e^2 + g_\rho(e_3,e_7) \,e^3 + g_\rho(e_4,e_7) \,e^4.$
So, by \eqref{def:Lie}, 
$$
de_7^\flat\ = g_\rho(e_3,e_7) \, de^3 + g_\rho(e_4,e_7) \, de^4 =  g_\rho(e_3, e_7)\, e^{12} + g_\rho(e_4,e_7)\,e^{13}.
$$
Therefore,
$$
\begin{aligned}
& g_\rho(d e_7^\flat, de_7^\flat) = g_\rho(e_3,e_7)^2  g_\rho(e^{12}, e^{12}) + 2\, g_\rho(e_3, e_7) g_\rho(e_4,e_7)\, g_\rho(e^{12}, e^{13}) \\
& \,\,\,\,\,\, \,\,\,\, \,\,\,\, \,\,\,\, \,\,\,\, \,\,\,\, \,\,\,\, \,\,\,\, + g_\rho(e_4,e_7)^2 g_\rho(e^{13}, e^{13})\\
& \,\,\,\,\,\, \,\,\,\, \,\,\,\, \,\,\,\, \,\,\,\, \,\,\,\, \,\,\,\, \,\,\,\,  = - g_\rho(e_4,e_7)^2 g_\rho(e^{1}, e^{3})^2,
\end{aligned}
$$
where the last equality follows from Lemma \ref{lem:grhozero2}. In fact, the equalities $g_\rho(e^{1}, e^{1}) =0=g_\rho(e^{1}, e^{2})$ imply $g_\rho(e^{12}, e^{12}) =0=g_\rho(e^{12}, e^{13})$,
and $g_\rho(e^{13}, e^{13}) = g_\rho(e^{1}, e^{1}) g_\rho(e^{3}, e^{3}) - g_\rho(e^{1}, e^{3})^2 = - g_\rho(e^{1}, e^{3})^2$. 
Now, from \eqref{eqn:ric7} we have 
$$
g_\rho(d e_7^\flat, de_7^\flat) = - g_\rho(e_4,e_7)^2 g_\rho(e^{1}, e^{3})^2=0.
$$
Since $g_\rho$ is a non Ricci-flat Einstein pseudo-metric induced by a closed $\G_2^\ast$-structure on $\gg$, then $g_\rho(e_4,e_7)$ is nonzero 
by Lemma~\ref{lem:grhononzero}. So $g_\rho(e^1,e^3)$ vanishes.

To prove that $g_\rho(e^1,e^4)=0$, we proceed as in the previous proof that $g_\rho(e^1,e^3)=0$.
From Lemma~\ref{lem:grhozero1} we know that $g_\rho(e_6,e_6)=0$. Thus, $\Ric(g_\rho)(e_6,e_6)=0$
because the pseudo-metric $g_\rho$ is Einstein.
Then, by \eqref{eq:Riccicurvature},
\begin{equation}\label{eqn:ric6}
g_{\rho}(d e_6^\flat, de_6^\flat) - g_{\rho}(\ad(e_6),\ad(e_6)) =0.
\end{equation}
We determine each term on the left-hand side of \eqref{eqn:ric6}. We begin with $g_{\rho}(d e_6^\flat, de_6^\flat)$. 
 By Lemma~\ref{lem:grhozero1}, $g_\rho(e_i, e_6)=0$, for $i=6,7,$ but $g_\rho(e_5,e_6)\not=0$
by Lemma~\ref{lem:grhononzero}. Hence, $e_6^\flat\ = \iota_{e_6} g_\rho = g_\rho(e_1,e_6)\, e^1 + g_\rho(e_2,e_6) \,e^2 + g_\rho(e_3,e_6) \,e^3 + g_\rho(e_4,e_6) \,e^4.
+ g_\rho(e_5,e_6) \,e^5$. So, using \eqref{def:Lie}, 
$$
\begin{aligned}
& de_6^\flat\ = g_\rho(e_3,e_6) \, de^3 + g_\rho(e_4,e_6) \, de^4 + g_\rho(e_5,e_6) \,de^5 \\
& \,\,\,\,\,\, \,\,\,\,=  g_\rho(e_3, e_6)\, e^{12} + g_\rho(e_4,e_6)\,e^{13} + g_\rho(e_5,e_6) \,e^{14}.
\end{aligned}
$$
Therefore, 
$$
\begin{aligned}
& g_{\rho}(d e_6^\flat, de_6^\flat) = g_\rho(e_3,e_6)^2 \, g_\rho(e^{12}, e^{12})+ g_\rho(e_4,e_6)^2 \, g_\rho(e^{13}, e^{13}) \\
& \,\,\,\,\,\, \,\,\,\, \,\,\,\,\,\, \,\,\,\, \,\,\,\,\,\, \,\,\,\, + g_\rho(e_5,e_6)^2 \, g_\rho(e^{14}, e^{14})  + 2 \, g_\rho(e_3, e_6) \,g_\rho(e_4,e_6)\, g_\rho(e^{12}, e^{13}) \\
& \,\,\,\,\,\, \,\,\,\, \,\,\,\,\,\, \,\,\,\, \,\,\,\,\,\, \,\,\,\, + 2\, g_\rho(e_3, e_6) \, g_\rho(e_5,e_6) \, g_\rho(e^{12}, e^{14}) + 2 \,  g_\rho(e_4, e_6) \, g_\rho(e_5,e_6) \, g_\rho(e^{13}, e^{14}).
\end{aligned}
$$
From Lemma \ref{lem:grhozero2} we have $g_\rho(e^{12}, e^{1m}) =0$, for $m=2, 3, 4$. Moreover, Lemma \ref{lem:grhozero2} and the equality $g_\rho(e^1,e^3)=0$ imply that
$g_\rho(e^{13}, e^{1n}) =0$, for $n=3, 4$. Thus, taking into account that $g_\rho(e^{14}, e^{14}) = - g_\rho(e^{1}, e^{4})^2$, we obtain
\begin{equation}\label{eqn:ric6-6}
g_{\rho}(d e_6^\flat, de_6^\flat) =  - g_\rho(e_5,e_6)^2  \, g_\rho(e^{1}, e^{4})^2. 
\end{equation}
 On the other hand, from \eqref{def:Lie} it follows that $\ad(e_6)=e^1\otimes e_7$. Hence,
 $$
 g_{\rho}(\ad(e_6),\ad(e_6))= g_{\rho}(e^1\otimes e_7, e^1\otimes e_7) = g_{\rho}(e^1,e^1)\, g_{\rho}(e_7,e_7) =0,
 $$
 where in the last equality we use that $g_{\rho}(e_7,e_7)=0$ by Lemma~\ref{lem:grhozero1}.
Since $g_{\rho}(\ad(e_6),\ad(e_6)) = 0$, \eqref{eqn:ric6} becomes
 $$
 g_{\rho}(d e_6^\flat, de_6^\flat) = 0.
 $$
 Then, by \eqref{eqn:ric6-6}, we obtain
 $$
 g_\rho(e_5,e_6)^2g_\rho(e^1,e^4)^2 =0, 
 $$
 which implies $g_\rho(e^1,e^4)=0$ because $g_\rho$ is a non Ricci-flat Einstein pseudo-metric induced by a closed $\G_2^\ast$-structure on $\gg$, 
 and so $g_\rho(e_5,e_6)\not=0$ by Lemma~\ref{lem:grhononzero}.
\end{proof}

To prove the main result of this section, we will need also the following.

\begin{proposition}{\rm{(}}\cite[Proposition 2.4]{CR1}{\rm{)}} \label{pro:obstruction2}
Let  $\frh$ be a nilpotent Lie algebra and let $h$ be a pseudo-metric on $\frh$. Let $\ad(\frh)$ be the image of $\frh$ in $\gz(\frh)^\circ\otimes \frh^1,$
where $\frh^1 = [\frh, \frh]$, and $\gz(\frh)^\circ$ is the annihilator of the center $\gz(\frh)$ of $\frh$ in $\frh^*$. Let $d(\frh^\ast)$ be the image 
of $\frh^\ast$ in $\bigwedge^2 \gz(\frh)^\circ$. Let ${\mathcal M}$ be the null space of $\ad(\frh)$ and let ${\mathcal N}$ be the null space of $d(\frh^\ast)$. If
\begin{equation}\label{eq:dimMN}
\dim {\mathcal M}+\dim {\mathcal N} \geq \dim \frh^1 - \dim \gz(\frh),
\end{equation}
then $h$ is not Einstein unless it is Ricci-flat.
\end{proposition}

\begin{theorem}\label{thm:thmEinstein2}
The Lie algebra $\frg$ defined in \eqref{def:Lie} has a closed $\G_2^\ast$-structure,  
but $\frg$ does not admit any closed $\G_2^\ast$-structure $\rho$ such that the induced pseudo-metric $g_{\rho}$ is Einstein and non Ricci-flat.

In particular, the non Ricci-flat Einstein pseudo-metric given in  \eqref{def:einsteinmet} is not determined by any closed $\G_2^\ast$-structure
on $\frg$.
\end{theorem}

\begin{proof}
Let us consider the coframe $\{h^1,\ldots,h^7\}$ of $\frg^*$ given by
\begin{equation}\label{def:newbasis}
\begin{aligned}
& e^1 = \frac{1}{2} \left(-h^2+h^6\right), \quad e^2=-h^4+h^7, \quad  e^3 =h^3, \\
& e^4= -\frac{1}{2}\,h^1+h^3-h^4+\frac{1}{2}\,h^5+h^7, \quad \,\,\, e^5=h^1+\frac{1}{2}\,h^3+h^5, \\
& e^6 = h^2-2\,h^4+h^6+2\,h^7, \quad  \,\,\, e^7=\frac{1}{2} \left(4\,h^2-3\,h^4+5\,h^7\right).
\end{aligned}
\end{equation}
We define the $\G_2^*$-structure $\varphi$ on $\frg$ for which $\{h^1,\ldots,h^7\}$ is an adapted coframe, i.e. $\varphi$ is given by
$$
\varphi = - h^{127} - h^{347} + h^{567} + h^{135} - h^{146} - h^{236}  - h^{245}.
$$ 
One can check that the 3-form $\varphi$, in terms of the coframe $\{e^1,\ldots,e^7\}$ of $\frg^*$, has the following expression
\begin{equation}\label{def:G2*form}
\varphi=e^{137}+2\, e^{156}-2\, e^{157}+e^{235}-e^{237}+e^{246}+e^{345}.
\end{equation}
We claim that $\varphi$ is closed. Indeed, by the equations \eqref{def:Lie}, we have $d(e^{137} + e^{235}) = 0$,
$2\,d(e^{156} - e^{157}) = - 2\,e^{1245}$ and  $d(- e^{237} + e^{246} + e^{345}) = 2\,e^{1245}$. Thus, $\varphi$ 
defines a closed $\G_2^*$-structure on $\frg$.

Next, we are going to show that no closed $\G_2^*$-structure on $\frg$ induces a non Ricci-flat  Einstein pseudo-metric. 

Assume that  $\frg$ has a closed $\G_2^\ast$-structure $\rho$ such that the induced pseudo-metric $g_\rho$ is Einstein and non Ricci-flat,
i.e. $\Ric(g_\rho)=\lambda \, g_\rho$, for some nonzero constant. 
Firstly, we prove that the Lie algebra $\frg$, with the non Ricci-flat Einstein pseudo-metric $g_\rho$, satisfies 
 the inequality ~\eqref{eq:dimMN}. 

The right-side of~\eqref{eq:dimMN} can be computed directly from~\eqref{def:Lie}. We have 
$\gg^1=[\gg,\gg]= \langle e_3, e_4 ,e_5, e_6, e_7 \rangle$ and  $\gz(\frg) = \langle e_7 \rangle$. So, 
$\dim[\gg,\gg]-\dim \gz(\gg)=4$. 

Now, we show that $\dim {\mathcal M}\geq 2$ because $\ad(e_5), \ad(e_6)\in {\mathcal M}$. 
By  \eqref{def:Lie}, the space $\ad(\gg)$ is generated by
$\ad(\gg) \subset \gz(\frg)^\circ\otimes \frg^1 = \langle e^i \otimes e_j; \, 1 \leq i \leq 6 \, \text{ and }   3 \leq j \leq 7 \rangle$ is generated by 
\begin{equation}
\begin{split}\label{eq:adjointg}
\ad(e_1)&=-e^2\otimes e_3-e^3\otimes e_4-e^4\otimes e_5-e^5\otimes e_6-e^6\otimes e_7,\\
\ad(e_2)&= e^1\otimes e_3 -e^3\otimes e_6-e^3\otimes e_7-e^4\otimes e_7,\\
\ad(e_3)&= e^1 \otimes e_4  + e^2\otimes e_6  + e^2 \otimes e_7,\qquad \ad(e_4)= e^1 \otimes e_5 + e^2 \otimes e_7,\\
\ad(e_5)&= e^1\otimes e_6,\quad \,\,\,\,\  \ad(e_6)= e^1\otimes e_7,\quad \ad(e_7)=0.
\end{split}
\end{equation}
Clearly $\ad(e_5) = e^1\otimes e_6, \ad(e_6)= e^1\otimes e_7 \in \Lambda^{1}(e^1,e^2,e^3,e^4)\otimes \frg+\gg^\ast\otimes \langle e_6,e_7 \rangle$
since $\ad(e_5)$ and $\ad(e_6)$ belong to each of those terms.  Moreover, because $g_{\rho}$ is a non Ricci-flat Einstein pseudo-metric induced 
by a closed $\G_2^\ast$-structure on $\frg$, from Lemma~\ref{lem:grhozero2} and Lemma ~\ref{lem:grhozero3} we know that 
 $e^1$ is $g_{\rho}$-orthogonal to $e^1,e^2,e^3$ and to $e^4$. Then, 
 $\ad(e_5)$ and $\ad(e_6)$ are both $g_{\rho}$-orthogonal to $\Lambda^{1}(e^1,e^2,e^3,e^4)\otimes \frg$.
 Also Lemma~\ref{lem:grhozero1} implies that $e_6$ and $e_7$ generate an isotropic subspace of $\gg$ (with respect to the pseudo-metric $g_{\rho}$). 
 Thus,  $\ad(e_5)$ and $\ad(e_6)$ are both $g_{\rho}$-orthogonal to $\gg^\ast\otimes \langle e_6,e_7 \rangle$.
 So  $\ad(e_5)$ and $\ad(e_6)$ are also $g_{\rho}$-orthogonal to $\Lambda^{1}(e^1,e^2,e^3,e^4)\otimes \frg+\gg^\ast\otimes \langle e_6,e_7 \rangle$.
 But from~\eqref{eq:adjointg},
\[
\ad(\gg)\subseteq  \Lambda^{1}(e^1,e^2,e^3,e^4)\otimes \frg+\gg^\ast\otimes \langle e_6,e_7 \rangle.
\]
Therefore, $\ad e_5$ and $\ad e_6$ are in the null space of $\ad\gg$, i.e. $\ad(e_5), \ad(e_6) \in {\mathcal M}$, and hence $\dim {\mathcal M}\geq 2$.

We determine a lower bound for $\dim {\mathcal N}$ by proving that $de^3, de^4, de^5\in {\mathcal N}$. First, note that
$de^3, de^4, de^5\in\Lambda^1(e^1)\wedge \Lambda^1(e^2,e^3,e^4)$ by ~\eqref{def:Lie}. 
Using again that  $g_{\rho}$ is a non Ricci-flat Einstein pseudo-metric on $\frg$ determined by a closed $\G_2^\ast$-structure on $\frg$, 
from Lemma~\ref{lem:grhozero2} and Lemma ~\ref{lem:grhozero3}
we know that $e^1$ is $g_{\rho}$-orthogonal to $e^1,e^2,e^3$ and $e^4$, 
that is $g_\rho(e^1,e^i)=0$, for $1\leq i\leq 4$. Then,
$$
g_{\rho}(e^{1j}, e^{1k}) = g_{\rho}(e^{1}, e^{1})\, g_{\rho}(e^{j}, e^{k}) - g_{\rho}(e^{1}, e^{j}) \,g_{\rho}(e^{1}, e^{k}) =0,
$$
for $1\leq j\leq 4$ and $1 \leq k \leq 7$. So, $\Lambda^1(e^1)\wedge \Lambda^1(e^2,e^3,e^4)$ is $g_{\rho}$-orthogonal to
$\Lambda^1(e^1)\wedge \gg^\ast$. Moreover,
$$
g_{\rho}(e^{1j}, e^{pq}) = g_{\rho}(e^{1}, e^{p})\, g_{\rho}(e^{j}, e^{q}) - g_{\rho}(e^{1}, e^{q}) \, g_{\rho}(e^{j}, e^{p}) =0,
$$
for $2\leq j, p,q\leq 4$, which implies that $\Lambda^1(e^1)\wedge \Lambda^1(e^2,e^3,e^4)$ is $g_{\rho}$-orthogonal to $\Lambda^2(e^2,e^3,e^4)$.
Thus, $\Lambda^1(e^1)\wedge \Lambda^1(e^2,e^3,e^4)$ is $g_{\rho}$-orthogonal to $\Lambda^1(e^1)\wedge \gg^\ast \oplus  \Lambda^2(e^2,e^3,e^4)$,
because $\Lambda^1(e^1)\wedge \Lambda^1(e^2,e^3,e^4)$ is $g_{\rho}$-orthogonal to each of those terms.
But $d(\gg^\ast)\subseteq \Lambda^1(e^1)\wedge \gg^\ast \oplus  \Lambda^2(e^2,e^3,e^4)$.
Consequently, $de^3, de^4, de^5\in {\mathcal N}$, and hence $\dim {\mathcal N}\geq 3$.

Therefore, $\dim {\mathcal M}+\dim {\mathcal N} \geq 5 \geq \dim \frg^1 - \dim \gz(\frg) = 4$.
This shows that the Lie algebra $\frg$, with the pseudo-metric $g_\rho$, satisfies 
 the inequality ~\eqref{eq:dimMN}. Then, $g_\rho$ must be Ricci-flat by Proposition \ref{pro:obstruction2}, which is a contradiction with 
 our assumption on the pseudo-metric $g_\rho$.
\end{proof}

\begin{remark}
In section \ref{sec:ricci-flatclosedG2}, Example 1,  we show that the closed $\G_2^\ast$-structure $\varphi$ given in \eqref{def:G2*form} is harmonic, but not coclosed, and the 
induced pseudo-metric $g_{\varphi}$ is not Einstein.
\end{remark}

%%%%%%%%%%%%%%%%%%%%%%%%%%%%%%%%%%%%%%%%%%%%%%%%%%%%%%%
%%%%%%%%%%%%%%%%%%%%%%%%%%%%%%%%%%%%%%%%%%%%%%%%%%%%%%%
%%%%%%%%%%%%%%%%%%%%%%%%%%%%%%%%%%%%%%%%%%%%%%%%%%%%%%%
%%\section{Closed harmonic $\G_2^*$-structures}\label{sec:ricci-flatclosedG2}
\section{Closed and harmonic $\G_2^*$-structures inducing Einstein pseudo-metrics}\label{sec:ricci-flatclosedG2}
%%%%%%%%%%%%%%%%%%%%%%%%%%%%%%%%%%%%%%%%%%%%%%%%%%%%%%%
%%%%%%%%%%%%%%%%%%%%%%%%%%%%%%%%%%%%%%%%%%%%%%%%%%%%%%%

Here we prove that if a closed and harmonic $\G_2^*$-structure $\varphi$ on a 7-manifold (not necessarily compact)
is such that the induced pseudo-Riemannian metric $g_{\varphi}$ is Einstein, then $g_{\varphi}$ is Ricci-flat.
Moreover, we give two examples of nilpotent Lie algebras, each of them endowed with a closed and harmonic $\G_2^*$-structure which is not coclosed. 
In one of these examples the induced pseudo-metric is not Einstein whereas the induced pseudo-metric on the other example is Ricci-flat and non-flat.
First, we need some results about $\G_2^*$-structures on manifolds.

The representation theory for $\G_2^*$-structures is similar to that of $\G_2$-structures, so that if
$V$ is a real vector space of dimension 7 with a  $\G_2^*$-structure $\varphi$, then
$\G_2^*$ acts irreducibly on $V,$ and hence on $\Lambda^1(V^*)$  and $\Lambda^6(V^*),$ but 
$\G_2^*$ acts reducibly on  $\Lambda^{p}(V^*),$ for $2 \leq p \leq 5$. Therefore, 
a $\G_2^*$-structure $\varphi$ on a 7-dimensional manifold $M$ induces splittings of the bundles $\Lambda^{p}T^*M$ 
of $p$-forms on $M$ into direct summands, which we denote by $\Lambda^{p}_{r}T^*M$, where $r$ is
the rank of the bundle. Let  $\Omega^{p}_{r}(M)$ be the space of sections of $\Lambda^{p}_{r}T^*M$. 
Then, the space $\Omega^p(M)$ of differential $p$-forms on $M$ $(p=2,3)$ admits the following decomposition 
\[
\renewcommand\arraystretch{1.2}
\begin{array}{rcl}
\Omega^2(M) &=& \Omega^2_{7} (M) \oplus \Omega^2_{14} (M),\\
\Omega^3(M) &=& \Omega^3_{1} (M)\oplus \Omega^3_7 (M)\oplus \Omega^3_{27}(M),
\end{array}
\renewcommand\arraystretch{1}
\]
where
 \begin{equation}\label{def:decomposition}
\begin{aligned}
& \Omega_{7}^2(M) =  \{ \star_{\varphi}(\alpha  \wedge \star_{\varphi} \varphi)  \, \mid  \, \alpha \in \Omega^1 (M) \},\\
& \Omega_{14}^2(M)=  \{ \beta \in \Omega^2 (M) \, \mid  \, \beta \wedge \varphi = - \star_{\varphi} \beta \}\\
& \, \quad \,  \qquad  = \{ \beta \in \Omega^2 (M) \, \mid  \, \beta \wedge  \star_{\varphi}\varphi = 0 \}, \\
& \Omega_{1}^3(M)=  \{ f\,\varphi \, \mid  \, f \in \Omega^0 (M)\},\\
& \Omega_{7}^3(M)=  \{  \star_{\varphi}(\alpha  \wedge \varphi)  \, \mid  \, \alpha \in \Omega^1 (M) \},\\
& \Omega^3_{27}  (M)= \{ \gamma \in \Omega^3 (M) \, \mid  \, \gamma \wedge \varphi =0= \gamma \wedge \star_{\varphi} \varphi \}.
\end{aligned}
\end{equation}
The Hodge star $\star_{\varphi} \colon \Omega^p (M) \to  \Omega^{7-p} (M)$ gives the corresponding decompositions of 
$\Omega^4(M)$ and $\Omega^5(M)$.

If one applies the splittings of the 4- and 5-forms to the differentials $d \varphi$ and $d \star_{\varphi}\varphi$, one may show
(see \cite{Br}) that
\begin{gather}\label{torsion}
\begin{cases}
d\varphi= \tau_0\,\star_{\varphi} \varphi +3\,\tau_1\wedge\varphi+\star_{\varphi} \tau_3,\\
d \star_{\varphi} \varphi= 4\,\tau_1\wedge\star_{\varphi} \varphi-\star_{\varphi} \tau_2,
\end{cases}
\end{gather}
where $\tau_0\in\Omega^0 (M), \tau_1\in\Omega^1 (M), \tau_{2}\in\Omega_{14}^2(M)$ and $\tau_3\in\Omega_{27}^3(M)$.

Therefore, if $\varphi$ is a closed $\G_2^*$-structure on $M$, i.e. $d\varphi = 0$,  then 
\eqref{torsion} implies that $\tau_0, \tau_1$ and $\tau_3$ are all zero, so the only non-zero torsion form is $\tau_2$, which
we  call the \emph{torsion 2-form of} $\varphi$.
For the rest of the article, we write $\tau=\tau_2$ for simplicity. 
Thus, if $\varphi$ is a closed $\G_2^*$-structure, by \eqref{def:decomposition} and \eqref{torsion}, 
\begin{equation}\label{torsionclosedG2star}
d\varphi= 0, \quad d \star_{\varphi} \varphi=  - \star_{\varphi} \tau  = \tau\wedge \varphi,
\end{equation}
since $\tau  \in \Omega_{14}^2(M)$.
Moreover, note that it can happen that $\left\|\tau\right\|_{g_{\varphi}}=0$, 
but this does not imply that $\tau =0$ because $g_{\varphi}$ has signature $(3,4)$.

Let $M$ be a 7-dimensional manifold with a $\G_2^*$-structure $\varphi$. We remind that
a differential $k$-form $\alpha$ on $M$ is said to be {\em harmonic} if
$$
 \Delta_{\varphi}  \alpha = 0,
$$
where $\Delta_{\varphi}$ is the Hodge Laplacian operator of the pseudo-Riemannian 
metric $g_{\varphi}$ determined by $\varphi$, that is, $\Delta_{\varphi}  \alpha = (d \delta +  \delta d) \alpha$ with 
$\delta  \alpha = (-1)^{7(k+1)+1} \star_{\varphi} d\, \star_{\varphi}\alpha$.

To prove the main result of this section, we need the following proposition, which shows a formula for the scalar curvature of 
the pseudo-Riemannian metric associated to a closed $\G_2^*$-structure on $M$. 

\begin{proposition}\label{pro:scalarcurvature}
Let $\varphi\in \Omega^3 (M)$ be a closed $\G_2^*$-structure on a 7-dimensional manifold $M,$ with torsion 2-form $\tau\in \Omega^2_{14}(M)$,
and let $g_{\varphi}$ be the pseudo-Riemannian metric on $M$ determined  by $\varphi$. Then,
the scalar curvature $\mathrm{scal}(g_{\varphi})$ of $g_{\varphi}$ is given by
\begin{equation}\label{eq:scalarcurvature}
\mathrm{scal}(g_{\varphi})=-\frac{1}{2} g_{\varphi}(\tau, \tau). 
\end{equation} 
\end{proposition}
 
The equality \eqref{eq:scalarcurvature} was obtained by Bryant for closed $\G_2$-structures  (see \cite[(4.16) of Section 4]{Br}). 
The proof of \eqref{eq:scalarcurvature} for closed $\G_2^*$-structures follows step by step the one given by Bryant 
in  \cite{Br} for closed $\G_2$-structures.  That is why we omit it here. 

\begin{theorem}\label{th:harmonic-Ricci}
Let $M$ be a 7-dimensional manifold equipped with a closed and harmonic $\G_2^*$-structure $\varphi$,
and let $g_{\varphi}$ be the pseudo-Riemannian metric on $M$ determined by $\varphi$. Then, 
the scalar curvature $\mathrm{scal}(g_{\varphi})$ of $g_{\varphi}$ vanishes. Therefore, 
if $g_{\varphi}$ is Einstein, $g_{\varphi}$ is Ricci-flat.
\end{theorem}

\begin{proof}
By Proposition \ref{pro:scalarcurvature}, $\mathrm{scal}(g_{\varphi}) = 0$ if and only if 
$g_{\varphi}(\tau, \tau)=0$. To prove that $g_{\varphi}(\tau, \tau)=0$ we proceed as follows. 
From \eqref{torsionclosedG2star} we know that $d \star_{\varphi} \varphi=  - \star_{\varphi} \tau$, and so $\tau = - \star_{\varphi} d \star_{\varphi} \varphi$. Then, using that
$\varphi$ is closed and harmonic, we have that $\tau$ is closed. In fact, 
\begin{equation*}
d\tau =  - d\star_{\varphi} d\star_{\varphi} \varphi = d\delta \varphi = \Delta_{\varphi} \varphi = 0.
\end{equation*}
On the other hand,
\begin{equation}\label{eqn:norm}
g_{\varphi}(\tau,\tau)\,\vol_{\varphi} = \tau \wedge \star_{\varphi} \tau = - \tau \wedge d \star_{\varphi} \varphi ,
\end{equation} 
where in the last equality we use \eqref{torsionclosedG2star}. But by \eqref{def:decomposition}, $\tau \wedge \star_{\varphi} \varphi =0$ since $\tau\in \Omega^2_{14}(M)$. 
So, taking into account that $\tau$ is closed, we have
$$
0 = d(\tau \wedge \star_{\varphi} \varphi) = d\tau \wedge \star_{\varphi} \varphi + \tau \wedge d\star_{\varphi} \varphi =\tau \wedge d\star_{\varphi} \varphi.
$$
Therefore, $\tau \wedge d\star_{\varphi} \varphi=0$, and so  $g_{\varphi}(\tau,\tau)=0$ by \eqref{eqn:norm}. Hence, 
$\mathrm{scal}(g_{\varphi}) = 0$ by \eqref{eq:scalarcurvature}.

Let us now suppose that $g_{\varphi}$ is Einstein. Then, $g_{\varphi}$ is Ricci-flat
since an Einstein pseudo-Riemannian metric on $M$, with zero scalar curvature, is Ricci-flat.
This completes the proof of the theorem.
\end{proof}

Next, we show two examples of 7-dimensional nilpotent Lie algebras, each of them equipped with a closed 
and harmonic $\G_2^*$-structure, which is not coclosed. In the first example, the induced pseudo-metric by the $\G_2^*$-structure
is not Einstein, while in the second example the induced pseudo-metric is Ricci-flat and non-flat.

\subsection{Example 1}
Let $\frg$ be the 7-dimensional nilpotent Lie algebra defined by \eqref{def:Lie}. Let us consider the
closed $\G_2^*$-structure $\varphi$ given by \eqref{def:G2*form}. 
We have that $\varphi$ is non-coclosed. In fact, one can check that
$\star_\varphi(\varphi)=\frac{1}{2} e^{1236} - e^{1256} - e^{1267} + \frac{1}{2} e^{1346} + e^{1356} -2 e^{1357} + e^{1456} -\frac{1}{2} e^{2346} -\frac{1}{2} e^{2347} - e^{2356} + e^{2357} + e^{2456} - e^{2457}$,
and 
$$
d(\star_\varphi(\varphi))= -\frac{7}{2} e^{12345} -\frac{3}{2} e^{12346} + e^{12347} + 2 e^{12356} - e^{12357} - e^{12456} \not=0.
$$
Hence, $\varphi$ is not coclosed.  However, $\varphi$ is harmonic. In fact, since $\varphi$ is closed, we have
$\Delta_{\varphi}\varphi= - d(\star_{\varphi}d(\star_{\varphi}\varphi))$. But, the 2-form $\star_{\varphi}(d(\star_{\varphi}\varphi))$ is given by
$$
\star_{\varphi}(d(\star_{\varphi}\varphi)) = \frac{5}{2} e^{12} -\frac{3}{2} e^{13} + e^{14} + e^{15} - e^{23},
$$
which is clearly closed by \eqref{def:Lie}. Thus, $\Delta_{\varphi}\varphi= 0$.

Let $g_{\varphi}$ be the pseudo-metric associated to $\varphi$. Since $\varphi$ is a closed and harmonic $\G_2^*$-structure on $\frg$,
from Theorem \ref{th:harmonic-Ricci} we know that the scalar curvature $\mathrm{scal}(g_{\varphi})$
of $g_\varphi$ vanishes. Note that if  $\tau$  is the torsion 2-form of $\varphi$, 
we have $\tau = - \star_{\varphi}(d(\star_{\varphi}\varphi)) = - \big(\frac{5}{2} e^{12} -\frac{3}{2} e^{13} + e^{14} + e^{15} - e^{23}\big)$ and 
$\star_{\varphi} \tau= - d(\star_\varphi(\varphi)) = \frac{7}{2} e^{12345} + \frac{3}{2} e^{12346} - e^{12347} - 2 e^{12356} + e^{12357} + e^{12456}$, and so
$\tau \wedge \star_{\varphi} \tau =0$ or, equivalently, $\mathrm{scal}(g_{\varphi}) = 0$ by \eqref{eq:scalarcurvature}. 

Let $\{h_1, \ldots, h_7\}$ be 
the frame of $\frg$  dual to the coframe $\{h^1, \ldots, h^7\}$ defined in \eqref{def:newbasis}. 
Then,  with respect to the $g_{\varphi}$-orthonormal frame $\{h_1, \ldots, h_7\}$, 
we have $g_{\varphi} = \mathrm{diag}(-1,-1,-1,-1,1,1,1)$, while
\begin{equation*}
\Ric(g_{\varphi}) = \begin{pmatrix}%{ccccccc}
 0 & 0 & 0 & 0 & 0 & 0 & 0 \\
 0 & \frac{1}{4} & 0 & 0 & 0 & -\frac{1}{4} & 0 \\
 0 & 0 & 0 & 0 & 0 & 0 & 0 \\
 0 & 0 & 0 & -\frac{1}{2} & 0 & 0 & \frac{1}{2} \\
 0 & 0 & 0 & 0 & 0 & 0 & 0 \\
 0 & -\frac{1}{4} & 0 & 0 & 0 & \frac{1}{4} & 0 \\
 0 & 0 & 0 & \frac{1}{2} & 0 & 0 & -\frac{1}{2} \\
\end{pmatrix}.
\end{equation*}
Thus, $g_\varphi$ is not Einstein. We can show again that $\mathrm{scal}(g_{\varphi})=0$.  Indeed,
 $\mathrm{scal}(g_{\varphi}) = \sum_{k =1}^7 \, g_{\varphi}^{kk}\, \Ric(g_{\varphi})_{kk} =0$, 
where $g_{\varphi}^{-1} = (g_{\varphi}^{ij})$ is the inverse matrix of $g_\varphi$, so $g_{\varphi}^{-1} = \mathrm{diag}(-1,-1,-1,-1,1,1,1)$.

\subsection{Example 2} Let $\frn$ be the nilpotent Lie algebra of dimension 7 with structure equations
\begin{equation}\label{def:lie-n}
 (0, \, 0, \, e^{12}, \, 0, \, 0, \, e^{13}+e^{24}, \, e^{15}).
\end{equation}
Let us consider the coframe $\{x^1,\ldots,x^7\}$ of $\frn^*$ given by
\begin{equation}\label{def:newbasis-n}
\begin{aligned}
& x^1 = - e^7, \quad x^2=e^1-\tfrac{1}{2}\,e^6, \quad  x^3 =e^3-\tfrac{1}{2}\,e^4, \quad x^4=  \tfrac{1}{4}\,e^2-e^5,\\
& x^5= -e^3-\tfrac{1}{2}\,e^4, \quad  x^6 = \tfrac{1}{4}\,e^2+e^5, \quad  x^7=\e^1+\tfrac{1}{2}\,e^6.
\end{aligned}
\end{equation}
We define the $\G_2^*$-structure $\psi$ on $\frn$ for which $\{x^1,\ldots,x^7\}$ is an adapted coframe, i.e.
\begin{equation}\label{def:psi}
\begin{split}
\psi &= - x^{127} - x^{347} + x^{567} + x^{135} - x^{146} - x^{236}  - x^{245}\\
     &= e^{123}+\frac{1}{2} e^{257}+e^{167}+e^{347}-e^{456}
\end{split}
\end{equation}
This form $\psi$ is closed. In fact, $d(e^{123})=0$, $d(e^{257})=0$ and $d( e^{167} + e^{347}  - e^{456}) =0$ by \eqref{def:lie-n}. Thus $\psi$ defines a closed $\G_2^*$-structure on $\frn$.

Using \eqref{eqn:star-varphi} and \eqref{def:newbasis-n}, we have
\begin{equation}\label{exp:starpsi}
\begin{aligned}
\star_{\psi}\psi&= x^{1234}-x^{1256}-x^{3456}-x^{2467}+x^{2357}+x^{1457}+x^{1367} \\
&=- \tfrac{1}{2}e^{1256}- e^{1237}-e^{1346}-\tfrac{1}{2}e^{2345}-e^{4567}.
\end{aligned}
\end{equation}
Thus, $d(\star_\psi\psi)= - e^{13457} = \frac14\left(-x^{12345}-x^{12356}+x^{13457}+x^{13567}\right) \not=0.$
This means that $\psi$ is non-coclosed.  However, $\psi$ is harmonic. In fact,
\[
\star_\psi d(\star_\psi\psi)=\frac14\left(-x^{24}+x^{26}+x^{47}-x^{67}\right)=e^{15}.
\] 
But, since $\psi$ is closed, 
\begin{equation*}
\Delta_{\psi}\psi= - d(\star_{\psi}d(\star_{\psi}\psi))= - d(e^{15}) = 0,
\end{equation*}
i.e. $\psi$ is harmonic. Now, because $\psi$ is a closed and harmonic $\G_2^*$-structure on $\frn$,
from Theorem \ref{th:harmonic-Ricci} we know that the scalar curvature $\mathrm{scal}(g_{\psi})$
of $g_\psi$ vanishes. Note that if $\widetilde{\tau}$  is the torsion 2-form of $\psi$, 
we have $\widetilde{\tau} = - \star_{\psi}(d(\star_{\psi}\psi)) = - e^{15}$ and 
$\star_{\psi} \widetilde{\tau}= - d(\star_\psi(\psi)) = e^{13457}$, and so
$\widetilde{\tau} \wedge \star_{\psi} \widetilde{\tau} =0$ or, equivalently, $\mathrm{scal}(g_{\psi}) = 0$ by \eqref{eq:scalarcurvature}. 

Let $\{x_1, \ldots, x_7\}$ be the frame of $\frn$  dual to the coframe $\{x^1, \ldots, x^7\}$ defined in \eqref{def:newbasis-n}. 
Then, with respect to the $g_{\psi}$-orthonormal frame $\{x_1, \ldots, x_7\}$, 
we have $g_{\psi} = \mathrm{diag}(-1,-1,-1,-1,1,1,1)$. Let $R_{\psi}$ be the curvature tensor of the Levi-Civita connection of
$g_{\psi}$. We obtain 
$R_{\psi}(x_p, x_q, x_r, x_s) =0$,  excepting 
\begin{equation}\label{eqn:curvature-2}
\begin{aligned}
&R_{\psi}(x_2, x_4, x_2, x_4) = - R_{\psi}(x_2, x_4, x_2, x_6) = - R_{\psi}(x_2, x_4, x_4, x_7) \\ 
& \qquad \qquad \qquad  \quad \,\, = R_{\psi}(x_2, x_4, x_6, x_7) = R_{\psi}(x_2, x_6, x_2, x_6) \\
& \qquad \qquad \qquad  \quad \,\, = R_{\psi}(x_2, x_6, x_4, x_7) = - R_{\psi}(x_2, x_6, x_6, x_7) \\ 
& \qquad \qquad \qquad  \quad \,\,  = R_{\psi}(x_4, x_7, x_4, x_7) = - R_{\psi}(x_4, x_7, x_6, x_7) \\
& \qquad \qquad \qquad  \quad \,\,  = R_{\psi}(x_6, x_7, x_6, x_7) = - \frac{3}{64},
\end{aligned}
\end{equation}
and those which one gets by the symmetries of $R_{\psi}$. Taking into account that,
with respect to the frame $\{x_1, \ldots, x_7\}$, the matrix
 $g_{\psi} = \mathrm{diag}(-1,-1,-1,-1,1,1,1)$ is diagonal, 
we have that $\Ric(g_\psi)_{ij}=\sum_{k = 1}^7 \, g_\psi^{kk}(R_{\psi})_{ikjk}$. 
Now, using \eqref{eqn:curvature-2}, we obtain $\Ric(g_\psi)_{ij} =0$ $(1\leq i, j \leq 7)$, i.e. 
$\Ric(g_\psi) =0$. Hence, $g_\psi$ is Ricci-flat but not flat.

\subsection*{Acknowledgements}
We are grateful to Diego Conti and Anna Fino for useful comments. The first and third authors were partially supported 
by the Basque Government Grant IT1094-16 and by the Grant PGC2018-098409-B-100 of 
the Ministerio de Ciencia, Innovaci\'on y Universidades of Spain.
The second author was partially supported by a \emph{For\-schungs\-stipendium} (FR 3473/2-1) from the Deutsche Forschungsgemeinschaft (DFG).

\end{document}